\documentclass[a4paper, 12pt]{amsart}
\usepackage[utf8]{inputenc}
\usepackage[english]{babel}
\usepackage{amsthm}
\usepackage{amsmath}
\usepackage{amsfonts}
\usepackage{mathtools}
\usepackage{enumitem}
\usepackage{xcolor}
\usepackage{stmaryrd}
\usepackage{amssymb}
\usepackage{mathrsfs}

\usepackage[margin=1.25in]{geometry}

\setlist[enumerate,1]{label = (\roman*)}

\theoremstyle{definition}
\newtheorem{definition}{Definition}[section]

\newtheorem{theorem}[definition]{Theorem}
\newtheorem{lemma}[definition]{Lemma}
\newtheorem{corollary}[definition]{Corollary}
\newtheorem{fact}[definition]{Fact}
\newtheorem*{remark}{Remark}

\newcommand{\EF}[3]{\mathrm{EF}^{#1}_{#2}\left(  #3  \right)}
\newcommand{\EFD}[3]{\mathrm{EFD}_{#1}^{#2}\left( #3 \right)}
\newcommand{\wins}{\mathbin{\raisebox{0.8\depth}{$\uparrow$}}}

\newcommand{\playerone}{\mathrm{\mathbf{I}}}
\newcommand{\playertwo}{\mathrm{\mathbf{II}}}

\newcommand{\A}{\mathfrak{A}}
\newcommand{\B}{\mathfrak{B}}

\newcommand{\N}{\mathbb{N}}
\newcommand{\R}{\mathbb{R}}

\newcommand{\abs}[1]{\left| #1 \right|}
\newcommand{\norm}[1]{\left\| #1 \right\|}

\DeclareMathOperator{\dom}{dom}
\DeclareMathOperator{\ran}{ran}
\DeclareMathOperator{\qr}{qr}
\DeclareMathOperator{\density}{density}

\DeclareMathOperator{\ar}{ar}
\DeclareMathOperator{\modulus}{mod}
\DeclareMathOperator{\bound}{bd}

\DeclareMathOperator{\id}{id}

\newcommand{\infinitary}[1]{\mathcal{L}_{#1\omega}}

\makeatletter
\newcommand{\dotminus}{\mathbin{\text{\@dotminus}}}
\newcommand{\@dotminus}{%
\ooalign{\hidewidth\raise1ex\hbox{.}\hidewidth\cr$\m@th-$\cr}%
}
\makeatother

\newcommand{\CFO}{\mathrm{CFO}} 
\DeclareMathOperator{\Th}{Th}
\DeclareMathOperator{\str}{str}

\title{Ehrenfeucht--Fraïssé Games for Continuous First-Order Logic}
\author{Åsa Hirvonen}
\email{asa.hirvonen@helsinki.fi}
\author{Joni Puljujärvi}
\email{joni.puljujarvi@helsinki.fi}
\date\today
\address{Department of Mathematics and Statistics, University of Helsinki, P.O. Box 68 (Pietari Kalmin katu 5), 00014 Helsinki, Finland}
\thanks{This project has received funding from the European Research Council (ERC) under the European Union’s Horizon 2020 research and innovation programme (grant agreement No 101020762).}

\begin{document}

\begin{abstract}
    We define a version of the Ehrenfeucht--Fraïssé game in the setting of metric model theory and continuous first-order logic and show that the second player having a winning strategy in a game of length $n$ exactly corresponds to being elementarily equivalent up to quantifier rank $n$. We then demonstrate the usefulness of the game with some examples. Finally, we discus connections between the game of length $\omega$ and infinitary logic.
\end{abstract}

\maketitle

\section{Introduction}

\noindent One of the important classical results of finite model theory is the connection between Ehrenfeucht--Fraïssé game of length $n$ and elementary equivalence up to quantifier rank $n$. In a finite relational signature, two structures satisfy the same sentences of quantifier rank $\leq n$ if and only if the second player has a winning strategy in the EF game of length $n$ between the structures. If the signature is not finite, then one needs to restrict the plays of the game to finite pieces of the signature. If the signature is not relational, then the connection between the length of the game and quantifier rank becomes more complicated and tied to nesting of function symbols.

In this paper, we replicate the classical connection between elementary equivalence up to quantifier rank $n$ and length $n$ EF games for continuous logic. Essentially the same game has appeared in lecture notes of Hart~\cite{hartgames} but, at least to the knowledge of the authors, no systematic study of it has been conducted. Classically, a tool fine-tuned for quantifier rank allows for more (in)expressivity results than a tool that can only be used to show full elementary equivalence, and such is the case also in the realm of metric structures.

In the last section, we discus infinite and dynamic versions of the game and connect it to infinitary logic. The connections between games and infinitary extensions of continuous first-order logic turn out not to be as straightforward as the connection between first-order continuous logic and games of finite length.

\section{Basic Definitions and Facts}

\subsection{Moduli of Uniform Continuity}

In the study of the semantics of continuous first-order logic, one important concept is that of a modulus of uniform continuity for a given function between metric spaces. At least two variants are found in the literature:
\begin{enumerate}
    \item A function $\Delta\colon[0,\infty)\to[0,\infty)$ is a modulus for a function $f\colon M\to M'$, where $M$ and $M'$ are metric spaces, if for every $\varepsilon>0$ and $x,y\in M$,
    \[
        d(x,y)<\Delta(\varepsilon) \implies d(f(x),f(y))\leq\varepsilon.
    \]
    \item A function $\Delta\colon[0,\infty)\to[0,\infty)$ is a modulus for a function $f\colon M\to M'$ if $\Delta(0) = 0$, $\Delta$ is continuous at $0$ and for all $x,y\in M$, we have
    \[
        d(f(x),f(y)) \leq \Delta(d(x,y)).
    \]
\end{enumerate}
It is easy to see that the former kind of a modulus for $f$ exists if and only if the latter kind of a modulus for $f$ exists if and only if $f$ is uniformly continuous. We adopt the latter formulation of a modulus, as it is more convenient in our setting. We may also make more assumptions about $\Delta$ to make it behave more nicely without losing generality. If $\Delta$ is a modulus for $f$, then by setting
\[
    \Delta'(t) = \inf\{ at + b \mid \text{$a,b\in[0,\infty)$ and $\Delta(s) \leq as+b$ for all $s\in[0,\infty)$} \}
\]
we obtain another modulus $\Delta'$ for $f$, with the additional properties that $\Delta'$ is continuous (everywhere), non-decreasing and also subadditive. We will be refering to moduli $\Delta$ with such properties without a mention to any particular function $f$ whose modulus $\Delta$ would be. This leads to the following definition.

\begin{definition}
    Let $\Delta\colon[0,\infty)\to[0,\infty)$ be a function.
    \begin{enumerate}
        \item We call $\Delta$ a \emph{modulus (of uniform continuity)} if
        \begin{enumerate}
            \item it vanishes at zero, i.e. $\Delta(0)=0$,
            \item it is non-decreasing and subadditive, i.e.
            \[
                \Delta(x)\leq\Delta(x+y)\leq\Delta(x)+\Delta(y)
            \]
            for all $x,y\in[0,\infty]$, and
            \item it is continuous.
        \end{enumerate}
        \item If $f\colon M\to M'$ is a function between two metric spaces, then we say that $f$ \emph{obeys} or \emph{respects} $\Delta$ if
        \[
            d(f(x),f(y))\leq\Delta(d(x,y))
        \]
        for all $x,y\in M$.
    \end{enumerate}
\end{definition}

\subsection{Metric Structures and Continuous First-order Logic}

Next we briefly recall the framework of metric structures, as well as continuous first-order logic $\CFO$. A good exposition is provided by \cite{MR2436146}. A metric signature $\sigma$ is a set of predicate, function and constant symbols together with the following additional information.
\begin{itemize}
    \item To each predicate and function symbol $s$ we assign a natural number $\ar(s)$, called the arity of $s$.
    \item To each predicate and function symbol $s$ we assign a modulus $\modulus(s)$.
    \item To each predicate symbol $P$, we assign a compact interval $\bound(P)$. Additionally, the signature specifies a number $D\in[0,\infty)$, intended to be an upper bound for the diameter of $\sigma$-structures.
\end{itemize}

\begin{definition}
    Let $\sigma$ be a signature. A $\sigma$-structure $\A$ is a complete metric space $(A,d)$ with $d(A)\leq D$, together with interpretations $s^\A$ of symbols $s\in\sigma$, satisfying the following.
    \begin{enumerate}
        \item For every predicate symbol $P$, $P^\A$ is a function $A^{\ar(P)}\to\bound(P)$ that obeys the modulus $\modulus(P)$.
        \item For every function symbol $f$, $f^\A$ is a function $A^{\ar(f)}\to A$ that obeys the modulus $\modulus(f)$.
        \item For every constant symbol $c$, $c^\A$ is an element of $A$.
    \end{enumerate}
    We denote by $\dom(\A)$ either the set $A$ or the metric space $(A,d)$, depending on the context. We will also simply write $a\in\A$ when we mean $a\in A$. We denote by $\str(\sigma)$ the class of all $\sigma$-structures.
\end{definition}

If $\sigma'\subseteq\sigma$, $\A$ is a $\sigma$-structure and $\B$ is a $\sigma'$-structure, we call $\B$ the $\sigma'$-reduct of $\A$, call $\A$ the $\sigma$-expansion of $\B$ and write $\B = \A\restriction_{\sigma'}$ if $\dom(\B) = \dom(\A)$ and $s^\B = s^\A$ for all $s\in\sigma'$. If $a_0,\dots,a_{n-1}\in\A$, then by $(\A,a_0,\dots,a_{n-1})$ we denote the $\sigma\cup\{c_0,\dots,c_{n-1}\}$-expansion of $\A$ for fresh constant symbols $c_i$, $i<n$, such that $c_i^{(\A,a_0,\dots,a_{n-1})} = a_i$.

\begin{definition}
    Let $\sigma$ be a signature. We fix a countably infinite set $\{v_i \mid i<\omega\}$ of variable symbols. The set of $\sigma$-formulae of $\CFO$ is defined as follows.
    \begin{enumerate}[label=(\roman*)]
        \item If $t$ and $t'$ are $\sigma$-terms, then $d(t,t')$ is a(n atomic) $\sigma$-formula, and if $P\in\sigma$ is an $n$-ary predicate symbol and $t_0,\dots,t_{n-1}$ are $\sigma$-terms, then $P(t_0,\dots,t_{n-1})$ is a(n atomic) $\sigma$-formula. The set of $\sigma$-terms is defined as in ordinary first-order logic.
        \item If $\varphi_0,\dots,\varphi_{n-1}$ are $\sigma$-formulae and $u\colon\prod_{i<n}\bound(\varphi_i)\to\R$ is a continuous function, then $u(\varphi_0,\dots,\varphi_{n-1})$ is a $\sigma$-formula. In this context we call $u$ an $n$-ary connective.
        \item If $\varphi$ is a $\sigma$-formula and $x$ is a variable, then $\inf_x\varphi$ and $\sup_x\varphi$ are $\sigma$-formulae.
    \end{enumerate}
\end{definition}

We denote by $\CFO[\sigma]$ the set of all $\sigma$-formulae. The set of free variables of a formula is defined as in classical first-order logic, with $\inf$ and $\sup$ acting as quantifiers. A formula with no free variables is called a sentence. If the free variables of $\varphi$ are among $x_0,\dots,x_{n-1}$, we write $\varphi(x_0,\dots,x_{n-1})$. Similarly, if the variables occuring in a term $t$ are among $x_0,\dots,x_{n-1}$, we write $t(x_0,\dots,x_{n-1})$. The \emph{quantifier rank} $\qr(\varphi)$ of a formula $\varphi$ counts the nesting of quantifiers and is defined recursively as follows. If $\varphi$ is atomic, then $\qr(\varphi) = 0$. If $\varphi = u(\psi_0,\dots,\psi_{n-1})$, then $\qr(\varphi) = \max_{i<n}\qr(\psi_i)$. If $\varphi = Q_x\psi$ with $Q\in\{\inf{},\sup{}\}$, then $\qr(\varphi) = \qr(\psi)+1$.

\begin{definition}
    Let $\sigma$ be a signature and $\A$ a $\sigma$-structure.
    \begin{enumerate}[label=(\Roman*)]
        \item If $t(x_0,\dots,x_{n-1})$ is a $\sigma$-term, the interpretation $t^\A$ of $t$ in $\A$ is defined as in ordinary first-order logic. More so than in first-order logic, we think of $t^\A$ as an $n$-ary function on $\A$. It follows that $t^\A$ is also uniformly continuous and respects a modulus that can be easily computed from the moduli of function symbols occurring in $t$.
        \item If $\varphi(x_0,\dots,x_{n-1})$ is a $\sigma$-formula, the interpretation $\varphi^\A$ of $\varphi$ in $\A$ is a function $\A^n\to\R$ defined as follows.
        \begin{enumerate}[label=(\roman*)]
            \item If $\varphi = d(t_0(x_0,\dots,x_{n-1}),t_1(x_0,\dots,x_{n-1}))$ for some $\sigma$-terms $t_0$ and $t_1$, then $\varphi^\A(a_0,\dots,a_{n-1}) = d(t_0^\A(a_0,\dots,a_{n-1}),t_1^\A(a_0,\dots,a_{n-1}))$ for all $a_0,\dots,a_{n-1}\in\A$.
            \item If $\varphi = P(t_0(x_0,\dots,x_{n-1}),\dots,t_{m-1}(x_0,\dots,x_{n-1}))$ for an $m$-ary predicate symbol $P\in\sigma$ and $\sigma$-terms $t_0,\dots,t_{m-1}$, then
            \[
                \hspace{5em} \varphi^\A(a_0,\dots,a_{n-1}) = P^\A(t_0^\A(a_0,\dots,a_{n-1}),\dots,t_{m-1}^\A(a_0,\dots,a_{n-1}))
            \]
            for all $a_0,\dots,a_{n-1}\in\A$.
            \item If $\varphi = u(\psi_0(x_0,\dots,x_{n-1}),\dots,\psi_{m-1}(x_0,\dots,x_{n-1})$ for an $m$-ary connective $u$ and $\sigma$-formulae $\psi_i$, $i<m$, then
            \[
                \hspace{5em} \varphi^\A(a_0,\dots,a_{n-1}) = u(\psi_0^\A(a_0,\dots,a_{n-1}),\dots,\psi_{m-1}^\A(a_0,\dots,a_{n-1}))
            \]
            for all $a_0,\dots,a_{n-1}\in\A$.
            \item If $\varphi = \inf_{x_n}\psi(x_0,\dots,x_n)$, then
            \[
                \hspace{5em} \varphi^\A(a_0,\dots,a_{n-1}) = \inf_{a_n\in\A}\psi^\A(a_0,\dots,a_n)
            \]
            for all $a_0,\dots,a_{n-1}\in\A$. The interpretation of $\sup_{x_n}\psi(x_0,\dots,x_n)$ is analogous.
        \end{enumerate}
    \end{enumerate}
\end{definition}

Given a modulus $\Delta$, we say that a formula $\varphi$ obeys $\Delta$ if $\varphi^\A$ obeys $\Delta$ for every structure $\A$. Every formula $\varphi$ has a modulus that it obeys and the least such modulus, which we denote by $\modulus(\varphi)$, can be computed from the moduli $\modulus(s)$ of predicate and function symbols $s\in\sigma$, as well as the moduli of connectives, as follows.
\begin{itemize}
    \item If $t$ is a $\sigma$-term, then a modulus $\modulus(t)$ is defined as follows: if $t$ is a variable or a constant symbol, then $\modulus(t) = 0$, and if $t = f(t_0,\dots,t_{n-1})$ for an $n$-ary function symbol $f\in\sigma$, then $\modulus(t)$ is the function $r\mapsto\modulus(f)(\max_{i<n}\modulus(t_i)(r))$.
    \item If $\varphi = P(t_0,\dots,t_{n-1})$ for an $n$-ary predicate symbol $P\in\sigma\cup\{d\}$ $\sigma$-terms $t_0,\dots,t_{n-1}$, then $\modulus(\varphi)$ is the function $r\mapsto\modulus(P)(\max_{i<n}\modulus(t_i)(r))$.
    \item If $\varphi = u(\varphi_0,\dots,\varphi_{n-1})$ for an $n$-ary connective $u$, then $\modulus(\varphi)$ is the function $r\mapsto\Delta_u(\max_{i<n}\modulus(\varphi_i)(r))$, where
    \[
        \Delta_u(r) = \sup\{ \abs{u(x_0,\dots,x_{n-1}) - u(y_0,\dots,y_{n-1})} : \max_{i<n}\abs{x_i - y_i}\leq r \}
    \]
    for all $r\in[0,\infty)$.
    \item If $\varphi = Qx\psi$ for $Q\in\{\inf,\sup\}$, then $\modulus(\varphi) = \modulus(\psi)$.
\end{itemize}

Similarly, each formula $\varphi(x_0,\dots,x_{n-1})$ has a bound, i.e. a compact interval $I$, such that for every structure $\A$, $\ran(\varphi^\A)\subseteq I$. For simplicity (and without loss of generality), we will henceforth assume that the signature $\sigma$ is such that $D = 1$ and $\bound(P) = [0,1]$ for all predicate symbols $P$ and $n$-ary connectives are functions $[0,1]^n\to[0,1]$. Thus the bound for any formula is $[0,1]$.

Two $\sigma$-formulae $\varphi(x_0,\dots,x_{n-1})$ and $\psi(x_0,\dots,x_{n-1})$ are \emph{logically equivalent} if
\[
    \varphi^\A(a_0,\dots,a_{n-1}) = \psi^\A(a_0,\dots,a_{n-1})
\]
for all $\sigma$-structures $\A$ and $a_0,\dots,a_{n-1}\in\A$. Note that with the convention that every formula is bounded by $[0,1]$, $\sup_x\varphi$ is logically equivalent to $1\dotminus\inf_x(1\dotminus\varphi)$, where $x\dotminus y = \max\{0, x-y\}$.

One thing to remark is that given, say, two unary connectives $u$ and $u'$ and a formula $\varphi$, the two formulae $u(u'(\varphi))$ and $(u\circ u')(\varphi)$ are not, as strings, the same, but they are of course logically equivalent.

Next we define the notion of a $\Delta$-formula for a given modulus $\Delta$. This will play an important role in the proof of our main result Theorem~\ref{Main theorem}.

\begin{definition}
    Given a modulus $\Delta$, we define the set of $\Delta$-formulae as follows.
    \begin{enumerate}
        \item If $\varphi$ is atomic, then it is a $\Delta$-formula if $\modulus(\varphi) \leq \Delta$.
        \item If $\varphi$ is a $\Delta$-formula, then so are $\sup_x\varphi$ and $\inf_x\varphi$.
        \item If $u$ is an $n$-ary connective that either respects $\Delta$ or is $1$-Lipschitz and for each $i<n$, $\varphi_i$ is either an atomic $\Delta$-formula or a $\Delta$-formula of the form $Qx\psi$ for $Q\in\{\sup,\inf\}$, then $u(\varphi_0,\dots,\varphi_{n-1})$ is a $\Delta$-formula.
    \end{enumerate}
\end{definition}

A $\Delta$-sentence is a $\Delta$-formula with no free variables.

Note that a $\Delta$-formula does not necessarily obey $\Delta$, but the intention is that $\Delta$-formulae of quantifier rank $\leq n$ obey $\Delta^n$, i.e. we tie the steepness provided by connectives (as measured by $\Delta$) to the quantifier rank of the formula.

We immediately make the following basic observation.

\begin{lemma}
    Every formula $\varphi$ is equivalent to a $\Delta$-formula of the same quantifier rank for some modulus $\Delta$. Moreover, the equivalent formula can be constructed simply by composing consequtive connectives in $\varphi$.
\end{lemma}

\subsection{Partial Isomorphisms and Elementary Equivalence}

\begin{definition}
    Let $\varepsilon \geq 0$, let $\Delta$ be a modulus, and let $\A$ and $\B$ be $\sigma$-structures.
    \begin{enumerate}
        \item A partial function $f\colon\A\to\B$ is called a partial $(\varepsilon,\Delta)$-iso\-mor\-phism if for all atomic $\Delta$-formulae $\varphi(x_0,\dots,x_{k-1})$ and $a_0,\dots,a_{k-1}\in\A$,
        \[
            \abs{\varphi^\A(a_0,\dots,a_{k-1}) - \varphi^\B(f(a_0),\dots,f(a_{k-1}))} \leq \varepsilon.
        \]

        \item A partial $\varepsilon$-isomorphism is a function which is a partial $(\varepsilon,\Delta)$-iso\-mor\-phism for every modulus $\Delta$. A partial isomorphism is a partial $0$-isomorphism. If the function is total and surjective, we drop the word ``partial''.
    \end{enumerate}
\end{definition}

Note that our definition of a partial isomorphism does not require that, whenever $f\colon\A\to\B$ is a partial isomorphism, $c^\A\in\dom(f)$ and $f(c^\A) = c^\B$ for all constant symbols $c\in\sigma$. However, if we do have $c^\A\in\dom(f)$, then for all $\varepsilon>0$,
\begin{align*}
    d(c^\B,f(c^\A)) = \abs{d(c^\A,c^\A) - d(c^\B,f(c^\A)} \leq \varepsilon,
\end{align*}
whence $f(c^\A) = c^\B$.

\begin{definition}
    Given $\varepsilon\geq 0$, a modulus $\Delta$, $n<\omega$ and $\sigma$-structures $\A$ and $\B$, we write $\A\equiv^{\varepsilon,\Delta}_n\B$ if
    \[
        \abs{\varphi^\A - \varphi^\B} \leq \varepsilon
    \]
    for all $\Delta$-sentences $\varphi$ of quantifier rank $\leq n$. We write $\A\equiv_n^\varepsilon\B$ if $\A\equiv^{\varepsilon,\Delta}_n\B$ for all $\Delta$. We write $\A\equiv_n\B$ if $\A\equiv_n^0\B$ and say in such a case that $\A$ and $\B$ are \emph{elementarily equivalent up to quantifier rank $n$}. We write $\A\equiv\B$ if $\A\equiv_n\B$ for all $n$ and say in such a case that $\A$ and $\B$ are \emph{elementarily equivalent}.
\end{definition}

\section{The Space of $\Delta$-formulae}

In this section we show that for a given modulus $\Delta$ and a natural number $n$, the set of $\Delta$-formulae with a fixed finite set of free variables and quantifier rank $\leq n$, equipped with the logical distance, is a \emph{totally bounded} pseudometric space. A (pseudo)metric space $M$ is totally bounded, or \emph{precompact}, if for any $\varepsilon>0$, $M$ can be covered with finitely many sets of diameter $<\varepsilon$. If $M$ is complete, then it is totally bounded if and only if it is compact.

Totally boundedness is, in a sense, approximate finiteness. A set of formulae of $\CFO$ being totally bounded with respect to the logical distance defined below corresponds to a set of formulae of classical first-order logic being finite up to logical equivalence. We begin by showing the set of atomic formulae is totally bounded exactly when the signature is finite and relational.

\begin{definition}
    The logical distance between two $\sigma$-formulae $\phi(\bar{x})$ and $\psi(\bar{x})$ with the same free variables is defined as
    \[
        d(\varphi,\psi) = \sup_{\A\in\str(\sigma)}\sup_{\bar{a}\in\A^k}\abs{\varphi^\A(\bar{a}) - \psi^\A(\bar{a})}.
    \]
\end{definition}

Note that two formulae are logically equivalent if and only if their logical distance is $0$.

\begin{lemma}
    Given $k<\omega$, the set of atomic $\sigma$-formulae $\varphi(v_0,\dots,v_{k-1})$ equipped with the logical distance is totally bounded if and only if $\sigma$ is relational and finite.
\end{lemma}
\begin{proof}
    We first show that if $\sigma$ contains a function symbol $f$, then the space is not totally bounded. Without loss of generality $f$ is unary. For each $n<\omega$, let $\A_n$ be a $\sigma$-structure whose domain is the $n$-element discrete metric space $\{0,\dots,n-1\}$ and $f^{\A_n}(i) = i+1 \mod n$. We define terms $t_i$, $i<\omega$, by setting $t_0 = v_0$ and $t_{i+1} = f(t_i)$. Then we let $\varphi_i(v_0) = d(t_0, t_i)$. Note that $t_i^{\A_n}(0) = (f^{\A_n})^i(0) = i \mod n$ for all $i$. Now for $i<j$,
    \begin{align*}
        d(\varphi_i,\varphi_j) &\geq \abs{\varphi_i^{\A_j}(0) - \varphi_j^{\A_j}(0)} \\
        &= \abs{d\left(t_0^{\A_j}(0),t_i^{\A_j}(0)\right) - d\left(t_0^{\A_j}(0),t_j^{\A_j}(0)\right)} \\
        &= \abs{d(0, i) - d(0, 0)} \\
        &= d(0,i) \\
        &= 1.
    \end{align*}
    Hence the space cannot be covered by finitely many balls of radius $1/2$.

    Then suppose that $\sigma$ is relational. We show that the space is discrete. Let $P,Q\in \sigma\cup\{d\}$ be two distinct predicates and denote $\varphi = P(t_0(\bar v),\dots,t_{n-1}(\bar v))$ and $\psi = Q(t'_0(\bar v),\dots,t'_{m-1}(\bar v))$ for some terms $t_i$ and $t'_j$. We let $\A$ be the two-element discrete space, making it a $\sigma$-structure by setting $P^\A \equiv 0$ and $R^\A \equiv 1$ for any $R\in\sigma\setminus\{P\}$ (or vice versa if it happens that $P = d$), whence
    \begin{align*}
        d(\varphi,\psi) &\geq \sup_{\bar a\in\A^k}\abs{P^\A(t_0^\A(\bar a),\dots,t_{n-1}^\A(\bar a)) - Q^\A((t'_0)^\A(\bar a),\dots,(t_{m-1})^\A(\bar a))} \\
        &= 1.
    \end{align*}

    Then let $\varphi = P(t_0,\dots,t_{n-1})$ and $\psi = P(t'_0,\dots,t'_{n-1})$, where for at least one $i<n$ we have $t_i\neq t'_i$. As $\sigma$ is relational, each term is either a constant symbol or a variable, and distinguishing between the two is irrelevant in this context, as it does not make a difference whether we define suitable interpretations for constants or find suitable parameters; hence we may assume they are all constants. Let $i<n$ be such that $t_i\neq t'_i$. Now let $\B$ be the two-element discrete space, with $t_i^\B = 0$, $(t'_i)^\B = 1$ and $P^\B(b_0,\dots,b_{i-1},0,b_{i+1},\dots,b_{n-1}) = 0$ and $P^\B(b_0,\dots,b_{i-1},1,b_{i+1},\dots,b_{n-1}) = 1$ for all $b_j\in\B$, $j\neq i$. Then
    \begin{align*}
        d(\varphi,\psi) &\geq \abs{P^\B(t_0^\B,\dots,t_{n-1}^\B) - P^\B((t'_0)^\B,\dots,(t'_{n-1})^\B)} \\
        &\geq 1.
    \end{align*}

    This shows that the space of all atomic formulae (in variables $v_0,\dots,v_{k-1}$) is discrete. It is clear that a discrete metric space is totally bounded if and only if it is finite.
\end{proof}

\begin{definition}
    For a signature $\sigma$, a modulus $\Delta$ and $n,k<\omega$, we denote by $X_{\sigma,\Delta,n,k}$ the set
    \[
        \{\varphi(v_0,\dots,v_{k-1})\in \CFO[\sigma] \mid \text{$\varphi$ is a $\Delta$-formula and $\qr(\varphi)\leq n$}\}
    \]
    equipped with the logical distance, making it a pseudometric space.
\end{definition}

In the following theorem, we show that the space $X_{\sigma,\Delta,n,k}$ is totally bounded for a finite and relational signature $\sigma$. This corresponds to the classical result that, up to logical equivalence, there are only finitely many formulae of classical first-order logic of quantifier rank $\leq n$ with a given finite set of free variables.

\begin{theorem}\label{Totally bounded formulae}
     Let $\sigma$ be a finite relational signature. For any modulus $\Delta$ and $n,k<\omega$, the space $X_{\sigma,\Delta,n,k}$ is totally bounded.
\end{theorem}
\begin{proof}
    Fix $\sigma$. We show by induction on $n$ that $X_{\sigma,\Delta,n,k}$ is totally bounded for all $k$ and $\Delta$. By the previous lemma, the space of all atomic $\sigma$-formulae with variables among $v_0,\dots,v_{k-1}$ is totally bounded, and thus, so is its subspace $X_{\sigma,\Delta,0,k}$. So suppose $X_{\sigma,\Delta,n,k}$ is totally bounded for all $\Delta$ and $k$, fix $\Delta$ and $k$ and look at $X_{\sigma,\Delta,n+1,k}$.
    Let $\varepsilon>0$. We wish to show that $X_{\sigma,\Delta,n+1,k}$ can be covered by finitely many sets of diameter $<\varepsilon$.
    
    Let $\delta>0$ be small enough that $\Delta(\delta)\leq\varepsilon/2$. By the induction hypothesis, $X_{\sigma,\Delta,n,k+1}$ is totally bounded so it has a finite cover $\{U_i \mid i<m\}$ with $d(U_i)<\delta$ for all $i<m$. Fix $\theta_i\in U_i$ for all $i<m$. By the Arzelà--Ascoli theorem\footnote{There are many versions of the Arzelà--Ascoli theorem. See e.g.~\cite{MR0070144}. The version we need here can be stated in the following form: if $X$ and $Y$ are compact metric spaces and $W$ is an equicontinuous set of functions $X\to Y$, then $W$ is totally bounded in the uniform topology of $C(X,Y)$.}, the set of all $m$-ary connectives $u\colon[0,1]^m\to[0,1]$ that either obey $\Delta$ or are $1$-Lipschitz is totally bounded, so it has a finite cover $\{V_i \mid i\in I\}$ with $d(V_i)<\varepsilon/2$ for all $i\in I$. Fix $u_i\in V_i$ for all $i\in I$. Let $\Sigma = \{u_i(\hat\theta_0,\dots,\hat\theta_{m-1}) \mid i\in I\}$, where $\hat\theta_i = \inf_{v_k}\theta_i$ for $i<m$. We show that $\Sigma$ is $\varepsilon$-dense in $X_{\sigma,\Delta,n+1,k}$, i.e. every $\varphi\in X_{\sigma,\Delta,n+1,k}$ is $\varepsilon$-close to some $\psi\in\Sigma$. As $\Sigma$ is a finite set of $\Delta$-formulae of quantifier rank $\leq n+1$, this proves the theorem.

    Now, let $\varphi\in X_{\Delta,n+1,k}$. The first thing to do is to massage $\varphi$ into a logically equivalent form $u(\hat\psi_0,\dots,\hat\psi_{l-1})$,
    where each $\hat\psi_i = \inf_{v_k}\psi_i$ for some $\psi_i\in X_{L,\Delta,n,k+1}$ and $u$ is a connective that either obeys $\Delta$ or is $1$-Lipschitz. If $\qr(\varphi)\leq n$, then $\varphi \equiv \id(\inf_{v_k}\varphi)$. If $\varphi = \inf_x\psi$ for some $\psi$, then by possibly changing names of bound variables we obtain $\varphi\equiv\id(\inf_{v_k}\psi^*)$ for some $\psi^*$. If $\varphi = \sup_x\psi$, then by again changing names of bound variables we obtain $\varphi\equiv\sup_{v_k}\psi^* \equiv 1\dotminus(\inf_{v_k}(1\dotminus\psi^*))$. As $t\mapsto 1\dotminus t$ is a $1$-Lipschitz connective, $1\dotminus\psi^*$ is equivalent to a $\Delta$-formula, and hence so is $\varphi$. For the connective case, we assume for simplicity of notation that $\varphi = u(Qx\psi)$ for unary $u$. By changing names of variables we may assume that $x = v_k$. If $Q = \inf$, we are done. Otherwise $\varphi \equiv u(1\dotminus(\inf_{v_k}(1\dotminus\psi)))$. Now again, since $t\mapsto 1\dotminus t$ is $1$-Lipschitz, $1\dotminus\psi$ is equivalent to a $\Delta$-formula. Left is to show that $t\mapsto u(1\dotminus t)$ is either $1$-Lipschitz or obeys $\Delta$. If $u$ is $1$-Lipschitz, then so is $u(1\dotminus t)$, so assume $u$ obeys $\Delta$. Then for all $t,s\geq 0$,
    \[
        \abs{u(1\dotminus t) - u(1\dotminus s)} \leq \Delta(\abs{(1\dotminus t) - (1\dotminus s)}).
    \]
    Now
    \begin{align*}
        \abs{(1\dotminus t) - (1\dotminus s)} &=
        \begin{cases}
            \abs{t-s} & \text{if $t,s\leq 1$}, \\
            \abs{1-t} &\text{if $t\leq 1 < s$}, \\
            \abs{1-s} &\text{if $s\leq 1 < t$ and} \\
            0 &\text{otherwise}, \\
        \end{cases}
    \end{align*}
    and in each case the expression is at most $\abs{t-s}$. Hence, as $\Delta$ is non-decreasing,
    \[
        \Delta(\abs{(1\dotminus t) - (1\dotminus s)}) \leq \Delta(\abs{t-s}),
    \]
    so we obtain
    \[
        \abs{u(1\dotminus t) - u(1\dotminus s)} \leq \Delta(\abs{t-s}).
    \]
    Thus $u(1\dotminus t)$ obeys $\Delta$.
    
    Let $\eta\colon l\to m$ be a function such that for each $j<l$, $\psi_j\in U_{\eta(j)}$; i.e. $\eta$ is an explicit choice function that picks for each $\psi_j$ a neighbourhood from the $\varepsilon$-cover $\{U_i \mid i<m\}$ of $X_{\sigma,\Delta,n,k+1}$. Without loss of generality, we may assume that $\eta$ is surjective.\footnote{If $\eta$ is not surjective, let $i_0,\dots,i_{p-1}$ enumerate $m\setminus\ran(\eta)$. Then define $\psi_{l+q} = \theta_{i_q}$ for all $q<p$. Now we extend $\eta$ into a function $\eta'\colon l+p\to m$ by setting $\eta(l+q) = i_q$ for all $q<p$. Then $\eta'$ is surjective, as $\psi_{l+q} = \theta_{i_q}\in U_{i_q} = U_{\eta'(l+q)}$ for all $q<p$. Then define a connective $u'\colon[0,1]^{l+p}\to[0,1]$ by
    \[
        u'(t_0,\dots,t_{l+p-1}) = u(t_0,\dots,t_{l-1})
    \]
    for all $t_0,\dots,t_{l+p-1}\in[0,1]$. Clearly also $u'$ obeys $\Delta$. Then we may replace the original formula by the equivalent formula $u'(\inf_{v_k}\psi_0,\dots,\inf_{v_k}\psi_{l+p-1})$, $l$ by $l+p$ and $\eta$ by $\eta'$.}
    As the order of the arguments of $u$ does not have an effect on whether $u$ obeys $\Delta$ or not\footnote{If $\pi$ is a permutation of the set $\{0,\dots,m-1\}$, then $u$ obeys $\Delta$ if and only if $u\circ(\Pr_{\pi(0)},\dots,\Pr_{\pi(m-1)})$ does.}, we may also assume that $\eta$ is non-decreasing.
    Then, define functions $c\colon m\to\omega$ and $\xi_i\colon[0,1]\to[0,1]^{c(i)}$, $i<m$, by $c(i) =\abs{\{j<l \mid \eta(j) = i\}}$ and $\xi_i(t) = (t,\dots,t)$. Then let $u^*\colon[0,1]^m\to[0,1]$ be defined by
    \[
        u^*(t_0,\dots,t_{m-1}) = u(\xi_0(t_0),\dots,\xi_{m-1}(t_{m-1})).
    \]
    Now $u^*(\hat\theta_0,\dots,\hat\theta_{m-1})$ is logically equivalent to the formula $u(\hat\theta_{\eta(0)},\dots,\hat\theta_{\eta(l-1)})$.
    Furthermore, $u^*$ obeys $\Delta$:
    \begin{align*}
        \hspace{2em}&\hspace{-2em}\abs{u^*(t_0,\dots,t_{m-1}) - u^*(t'_0,\dots,t'_{m-1})} \\
        &= \abs{u(\xi_0(t_0),\dots,\xi_{m-1}(t_{m-1})) - u(\xi_0(t'_0),\dots,\xi_{m-1}(t'_{m-1}))} \\
        &\leq \Delta(d((\xi_0(t_0),\dots,\xi_{m-1}(t_{m-1})),(\xi_0(t'_0),\dots,\xi_{m-1}(t'_{m-1})))) \\
        &= \Delta(\max_{i<m}d(t_i,t'_i)) \\
        &= \Delta(d((t_0,\dots,t_{m-1}), (t'_0,\dots,t'_{m-1})))
    \end{align*}
    for any $t_0,t'_0,\dots,t_{m-1},t'_{m-1}\geq 0$.

    Let $i\in I$ be such that $u^*\in V_i$ and denote $\chi = u_i(\hat\theta_0,\dots,\hat\theta_{m-1})$.
    To finish the proof, it suffices to show that $d(\varphi,\chi)<\varepsilon$.
    Now a standard argument involving infima gives that $d(\hat\psi_j,\hat\theta_{\eta(j)})\leq\delta$ for all $j<l$. From this and the fact that $u$ respects $\Delta$, it follows that
    \begin{align*}
        d(u&(\hat\psi_0,\dots,\hat\psi_{l-1}), u(\hat\theta_{\eta(0)},\dots,\hat\theta_{\eta(l-1)})) \\
        &= \sup_{\A,\bar a}\abs{u(\hat\psi_0^\A(\bar a),\dots,\hat\psi_{l-1}^\A(\bar a)) - u(\hat\theta_{\eta(0)}^\A(\bar a),\dots,\hat\theta_{\eta(l-1)}^\A(\bar a))} \\
        &\leq \sup_{\A,\bar a}\Delta\left(\max_{i<l}\abs{\hat\psi_i^\A(\bar a) - \hat\theta_{\eta(i)}^\A(\bar a)}\right) \\
        &\leq \Delta\left(\max_{i<l}\sup_{\A,\bar a}\abs{\hat\psi_i^\A(\bar a) - \hat\theta_{\eta(i)}^\A(\bar a)}\right) \\
        &\leq \Delta(\max_{i<l}d(\hat\psi_i,\hat\theta_{\eta(i)})) \\
        &\leq \Delta(\delta) < \frac{\varepsilon}{2}.
    \end{align*}
    On the other hand, by the choice of $i$, we obtain
    \begin{align*}
        d(u&(\hat\theta_{\eta(0)},\dots,\hat\theta_{\eta(l-1)}), u_i(\hat\theta_0,\dots,\hat\theta_{m-1})) \\
        &= \sup_{\A,\bar a}\abs{u(\hat\theta_{\eta(0)}^\A(\bar a),\dots,\hat\theta_{\eta(l-1)}^\A(\bar a)) -  u_i(\hat\theta_0^\A(\bar a),\dots,\hat\theta_{m-1}^\A(\bar a))} \\
        &\leq \sup_{\A,\bar a}\abs{u^*(\hat\theta_0^\A(\bar a),\dots,\hat\theta_{m-1}^\A(\bar a)) - u_i(\hat\theta_0^\A(\bar a),\dots,\hat\theta_{m-1}^\A(\bar a))} \\
        &\leq \sup_{\bar t\in[0,1]^m}\abs{u^*(\bar t) - u_i(\bar t)} \\
        &= \norm{u^* - u_i}_\infty < \frac{\varepsilon}{2}.
    \end{align*}
    Then we can conclude that
    \begin{align*}
        d(\varphi,\chi) &= d(u(\hat\psi_0,\dots,\hat\psi_{l-1}),u_i(\hat\theta_0,\dots,\hat\theta_{m-1})) \\
        &\leq d(u(\hat\psi_0,\dots,\hat\psi_{l-1}), u(\hat\theta_{\eta(0)},\dots,\hat\theta_{\eta(l-1)})) \\
        &\quad + d(u(\hat\theta_{\eta(0)},\dots,\hat\theta_{\eta(l-1)}), u_i(\hat\theta_0,\dots,\hat\theta_{m-1})) \\
        &< \frac{\varepsilon}{2} + \frac{\varepsilon}{2} = \varepsilon.
    \end{align*}
\end{proof}

\section{The Ehrenfeucht--Fraïssé Game}

\begin{definition}\label{Game definition}
    Let $\A$ and $\B$ be $\sigma$-structures with disjoint domains and $n<\omega$.
    \begin{enumerate}
        \item Let $\varepsilon>0$. The $\varepsilon$-Ehrenfeucht--Fraïssé game of length $n$ between $\A$ and $\B$, denoted by $\EF{\varepsilon}n{\A,\B}$, has two players, $\playerone$ and $\playertwo$, and is played as follows.
        On each round $i<n$, either
        \begin{enumerate}
            \item $\playerone$ plays an element $x_i\in\A$ and $\playertwo$ responds by playing $y_i\in\B$, or
            \item $\playerone$ plays $x_i\in\B$ and $\playertwo$ responds with a $y_i\in\A$.
        \end{enumerate}
        Then $\playertwo$ wins a play
        \[
            (x_0,y_0,\dots,x_{n-1},y_{n-1})
        \]
        if the function $\{(a_i,b_i) \mid i<n\}$ is a partial $\varepsilon$-isomorphism, where
        \[
            a_i =
            \begin{cases}
                x_i & \text{if $x_i\in\A$}, \\
                y_i & \text{otherwise},
            \end{cases}
            \quad\text{and}\quad
            b_i =
            \begin{cases}
                x_i & \text{if $x_i\in\B$}, \\
                y_i & \text{otherwise}.
            \end{cases}
        \]
        \item The Ehrenfeucht--Fraïssé game of length $n$ between $\A$ and $\B$, denoted simply by $\EF{}n{\A,\B}$, is played as follows. On the first round, $\playerone$ chooses $\varepsilon>0$, and then the players proceed to play $\EF{\varepsilon}n{\A,\B}$. The winner of this $\varepsilon$-play wins the play of $\EF{}n{\A,\B}$.
    \end{enumerate}
\end{definition}

A strategy for each player is defined as usual. We write $X\wins G$ if player $X$ has a winning strategy in the game $G$. It is immediately clear that $\playertwo\wins\EF{}n{\A,\B}$ if and only if $\playertwo\wins\EF{\varepsilon}n{\A,\B}$ for all $\varepsilon$.

Note that one could define a version of the game where, in the first round, in addition to choosing $\varepsilon>0$, $\playerone$ also chooses a modulus $\Delta$, and then the winning condition is that the function resulting from the moves of the players is a partial $(\varepsilon,\Delta)$-isomorphism. This could be useful in practice if considering signatures with function symbols to bound the number of formulae that must be considered at once while determining whether $\playertwo$ wins a play. However, as will be obvious in the proof of Theorem~\ref{Main theorem}, for the purposes of this paper this addition is useless for the finite game. In the case of an infinite variant of the game (see Section~\ref{Section: Infinite and Dynamic Games}), it is more useful but not quite enough for the purposes of infinitary logic.

The following lemma captures the essence of the power of our game.

\begin{lemma}\label{Essential lemma}
    Let $\sigma$ be any signature. For any formula $\varphi(v_0,\dots,v_{k-1})$, there is a modulus $\Theta$ such that for any $\sigma$-sturctures $\A$ and $\B$, tuples $\bar a\in\A^k$ and $\bar b\in\B^k$, and $\varepsilon>0$,
    \[
        \playertwo\wins\EF{\varepsilon}{\qr(\varphi)}{(\A,\bar a),(\B,\bar b)} \implies \abs{\varphi^\A(\bar a) - \varphi^\B(\bar b)}\leq\Theta(\varepsilon).
    \]
\end{lemma}
\begin{proof}
    We construct $\Theta$ by induction on the quantifier rank of $\varphi$. First notice that $\playertwo\wins\EF{\varepsilon}0{(\A,\bar a),(\B,\bar b)}$ implies $(\A,\bar a)\equiv_0^{\varepsilon}(\B,\bar b)$ by definition, so for atomic $\varphi$, we can choose $\Theta = \id$. If $\varphi$ is quantifier free but not atomic, then it is equivalent to $u(\psi_0(\bar v),\dots,\psi_{m-1}(\bar v))$ for some connective $u$ and atomic $\psi_i$. Then we may choose $\Theta$ to be any modulus respected by $u$, as $(\A,\bar a)\equiv_0^{\varepsilon}(\B,\bar b)$ implies
    \begin{align*}
        \abs{\varphi^\A(\bar a) - \varphi^\B(\bar b)} &= \abs{u(\psi_0^\A(\bar a),\dots,\psi_{m-1}^\A(\bar a)) - u(\psi_0^\B(\bar b),\dots,\psi_{m-1}^\B(\bar b))} \\
        &\leq \Theta\left(\max_{i<m}\abs{\psi_i^\A(\bar a) - \psi_i^\B(\bar b)}\right) \\
        &\leq \Theta(\varepsilon).
    \end{align*}

    Suppose that the claim holds for formulae of quantifier rank $n$, and fix a formula $\varphi(v_0,\dots,v_{k-1})$ of quantifier rank $\leq n+1$. Now $\varphi \equiv u(\hat\psi_0, \dots, \hat\psi_{m-1})$ for some connective $u$, where $\hat\psi_i = Q_i{v_k}\psi_i(v_0,\dots,v_k)$ for $Q_i\in\{\inf,\sup\}$ and each $\psi_i$ has quantifier rank $\leq n$. Let $\Delta$ be a modulus respected by $u$. By the induction hypothesis, there are moduli $\Theta_i$ such that
    \[
        \playertwo\wins\EF{\varepsilon}n{(\A,\bar a,a),(\B,\bar b, b)} \implies \abs{\psi_i^\A(\bar a, a) - \psi_i^\B(\bar b, b)} \leq \Theta_i(\varepsilon)
    \]
    for any $\A$, $\B$, $(\bar a, a)\in\A^{k+1}$, $(\bar b, b)\in\B^{k+1}$ and $\varepsilon > 0$. We declare
    \[
        \Theta(t) \coloneqq \Delta\left(\max_{i<m}\Theta_i(t)\right)
    \]
    for all $t\in[0,\infty)$. Clearly $\Theta$ is a modulus.
    We proceed to show that it satisfies the desired condition. For this, let $\A$, $\B$, $\bar a\in\A^k$, $\bar b\in\B^k$ and $\varepsilon>0$ be fixed, and suppose that $\playertwo\wins\EF{\varepsilon}{n+1}{(\A,\bar a),(\B,\bar b)}$.
    
    Now, for each $i<m$, we show that $|\hat\psi_i^\A(\bar a) - \hat\psi_i^\B(\bar b)| \leq \Theta_i(\varepsilon)$. Fix $i<m$. We assume that $Q_i = \inf$; the case $Q_i = \sup$ is similar. If $\hat\psi_i^\A(\bar a) = \hat\psi_i^\B(\bar b)$, we are done, so suppose this is not the case. By symmetry, we may assume that $\hat\psi_i^\B(\bar b) < \hat\psi_i^\A(\bar a)$. Let $\delta>0$ be arbitrary. Then there is $a\in\A$ such that $\psi_i^\A(\bar a, a) < \hat\psi_i^\A(\bar a) + \delta$. We now play $\EF{\varepsilon}{n+1}{(\A,\bar a),(\B,\bar b)}$. We let $\playerone$ play $a$ on the first round. Then the winning strategy of $\playertwo$ produces an element $b\in\B$. It follows that
    \[
        \playertwo\wins\EF{\varepsilon}n{(\A,\bar a,a),(\B,\bar b,b)}.
    \]
    By the induction hypothesis, we then have
    \[
        \abs{\psi_i^\A(\bar a,a) - \psi_i^\B(\bar b, b)} \leq \Theta_i(\varepsilon).
    \]
    Then
    \[
        \hat\psi_i^\B(\bar b) \leq \psi_i^\B(\bar b, b) \leq \psi_i^\A(\bar a, a) + \Theta_i(\varepsilon) < \hat\psi_i^\A(\bar a) + \Theta_i(\varepsilon) + \delta.
    \]
    As $\delta$ was arbitrary, this means that $|\hat\psi_i^\A(\bar a) - \hat\psi_i^\B(\bar b)| \leq \Theta_i(\varepsilon)$.

    Now,
    \begin{align*}
        \abs{\varphi^\A(\bar a) - \varphi^\B(\bar b)} &= \abs{u(\hat\psi_0^\A(\bar a), \dots, \hat\psi_{m-1}^\A(\bar a)) - u(\hat\psi_0^\B(\bar b), \dots, \hat\psi_{m-1}^\B(\bar b))} \\
        &\leq \Delta\left( \max_{i<m}\abs{\hat\psi_i^\A(\bar a) - \hat\psi_i^\B(\bar b)} \right) \\
        &\leq \Delta\left( \max_{i<m}\Theta_i(\varepsilon) \right) \\
        &= \Theta(\varepsilon).
    \end{align*}
    This concludes the proof.
\end{proof}

\begin{theorem}\label{Main theorem}
    Given $n<\omega$, consider the following conditions.
    \begin{enumerate}
        \item $\playertwo\wins\EF{}n{\A\restriction_{\sigma'},\B\restriction_{\sigma'}}$ for all  finite $\sigma'\subseteq\sigma$. \label{II wins}
        \item $\A\equiv_n\B$. \label{Elem equiv}
    \end{enumerate}
    We have \ref{II wins}$\implies$\ref{Elem equiv}, and if $\sigma$ is relational, then also \ref{Elem equiv}$\implies$\ref{II wins}.
\end{theorem}
\begin{proof}
    \ref{II wins}$\implies$\ref{Elem equiv}:
    Let $\sigma$ be a signature and $\A$ and $\B$ $\sigma$-structures, and suppose that $\playertwo\wins\EF{}n{\A\restriction_{\sigma'},\B\restriction_{\sigma'}}$ for all  finite $\sigma'\subseteq\sigma$. Then $\playertwo\wins\EF{\varepsilon}n{\A\restriction_{\sigma'},\B\restriction_{\sigma'}}$ for all $\varepsilon>0$ and finite $\sigma'\subseteq\sigma$. Suppose for a contradiction that $\A\not\equiv_n\B$. Then there is a sentence $\varphi$ of quantifier rank $\leq n$ and $\delta>0$ such that $\abs{\varphi^\A - \varphi^\B}\geq\delta$. Let $\sigma'$ be the set of symbols from $\sigma$ that occur in $\varphi$. Now by Lemma~\ref{Essential lemma}, there is a modulus $\Theta$ such that $\playertwo\wins\EF{\varepsilon}n{\A\restriction_{\sigma'},\B\restriction_{\sigma'}}$ implies $\abs{\varphi^\A - \varphi^\B}\leq\Theta(\varepsilon)$ for all $\varepsilon>0$. Let $\varepsilon$ be small enough that $\Theta(\varepsilon) < \delta$. Then, as $\playertwo\wins\EF{\varepsilon}n{\A\restriction_{\sigma'},\B\restriction_{\sigma'}}$, we have
    \begin{align*}
        \delta \leq \abs{\varphi^\A - \varphi^\B} &\leq \Theta(\varepsilon) < \delta, 
    \end{align*}
    a contradiction.

    \ref{Elem equiv}$\implies$\ref{II wins}:
    Let $\sigma$ be a relational signature and $\A$ and $\B$ $\sigma$-structures. Suppose that $\A\equiv_n\B$. Fix a finite $\sigma'\subseteq\sigma$ and a number $\varepsilon>0$. Now, as $\sigma'$ is finite and relational, there are only finitely many atomic formulae with variables among $v_0,\dots,v_{n-1}$. Hence there is a modulus $\Delta$ such that every atomic $\sigma'$-formula in variables $v_0,\dots,v_{n-1}$ is equivalent to a $\Delta$-formula. We define a strategy of $\playertwo$ in $\EF{\varepsilon}n{\A,\B}$ as follows. $\playertwo$ maintains the following condition: if the position is $(x_0,y_0,\dots,x_{i-i},y_{i-1})$, then
    \[
        (\A\restriction_{\sigma'},a_0,\dots,a_{i-1}) \equiv^{(i+1)\varepsilon/(n+1),\Delta}_{n-i} (\B\restriction_{\sigma'},b_0,\dots,b_{i-1}).
    \]
    We show that it is always possible for $\playertwo$ to do so. In the beginning, before any rounds have been played, we have $\A\restriction_{\sigma'}\equiv_n\B\restriction_{\sigma'}$ by assumption, whence $\A\restriction_{\sigma'}\equiv^{\varepsilon/(n+1),\Delta}_n\B\restriction_{\sigma'}$. Suppose that $\playertwo$ has been able to maintain the condition until the position is $(x_0,y_0,\dots,x_{i-1},y_{i-1})$. Suppose that $\playerone$ plays $a_i\in\A$ (the case where $\playerone$ plays in $\B$ is symmetric). By Lemma~\ref{Totally bounded formulae}, the space of $\Delta$-formulae of quantifier rank $\leq n-(i+1)$ with free variables among $v_0,\dots,v_i$ is totally bounded, so it has a finite $\varepsilon/3(n+1)$-cover $\{U_j \mid j\in J\}$. For each $j\in J$, let $\varphi_j(v_0,\dots,v_i)\in U_j$ and let $r_j = \varphi_j^\A(a_0,\dots,a_i)$. Then let $\psi(v_0,\dots,v_i) = \max_{j\in J}\abs{\varphi_j - r_j}$. Now $\inf_{v_i}\psi$ is (equivalent to) a $\Delta$-formula of quantifier rank $\leq n-i$, so as, by the induction hypothesis, we have
    \[
        (\A\restriction_{\sigma'},a_0,\dots,a_{i-1}) \equiv^{(i+1)\varepsilon/(n+1),\Delta}_{n-i} (\B\restriction_{\sigma'},b_0,\dots,b_{i-1}),
    \]
    we obtain
    \[
        \abs{ \inf_a\psi^\A(a_0,\dots,a_{i-1},a) - \inf_b\psi^\B(b_0,\dots,b_{i-1},b) } \leq\frac{(i+1)\varepsilon}{n+1},
    \]
    whence
    \begin{align*}
        \inf_b\psi^\B(b_0,\dots,b_{i-1},b) &\leq \inf_a\psi^\A(a_0,\dots,a_{i-1},a) + \frac{(i+1)\varepsilon}{n+1} \\
        &\leq \psi^\A(a_0,\dots,a_i) + \frac{(i+1)\varepsilon}{n+1} \\
        &= 0 + \frac{(i+1)\varepsilon}{n+1} \\
        &< \frac{(i+1)\varepsilon}{n+1} + \frac{\varepsilon}{3(n+1)}.
    \end{align*}
    Thus there is $b_i\in\B$ with
    \[
        \psi^\B(b_0,\dots,b_i) < \frac{(i+1)\varepsilon}{n+1} + \frac{\varepsilon}{3(n+1)}.
    \]
    Hence for all $j\in J$,
    \[
        \abs{\varphi_j^\B(b_0,\dots,b_i) - r_j} \leq \frac{(i+1)\varepsilon}{n+1} + \frac{\varepsilon}{3(n+1)}.
    \]
    Now, if $\varphi(v_0,\dots,v_i)$ is a $\Delta$-formula of quantifier rank $\leq n-(i+1)$, then for some $j\in J$, $\varphi\in U_j$ and thus $d(\varphi,\varphi_j) < \varepsilon/3(n+1)$. Then, denoting $\bar a = (a_0,\dots,a_i)$ and $\bar b = (b_0,\dots,b_i)$,
    \begin{align*}
        \abs{\varphi^\A(\bar a) - \varphi^\B(\bar b)} &\leq \abs{\varphi^\A(\bar a) - \varphi_j^\A(\bar a)} + \abs{\varphi_j^\A(\bar a) - \varphi_j^\B(\bar b)} + \abs{\varphi_j^\B(\bar b) - \varphi^\B(\bar b)} \\
        &\leq \frac{\varepsilon}{3(n+1)} + \left(\frac{(i+1)\varepsilon}{n+1} + \frac{\varepsilon}{3(n+1)}\right) + \frac{\varepsilon}{3(n+1)} \\
        &= \frac{(i+1)\varepsilon}{n+1} + \frac{\varepsilon}{n+1} \\
        &= \frac{(i+2)\varepsilon}{n+1}.
    \end{align*}
    Hence
    \[
        (\A\restriction_{\sigma'},a_0,\dots,a_{i-1},a) \equiv^{(i+2)\varepsilon/(n+1),\Delta}_{n-(i+1)} (\B\restriction_{\sigma'},b_0,\dots,b_{i-1}, b).
    \]
    Thus $\playertwo$ can maintain the condition.

    Clearly this strategy is winning, as at the end of a play the condition that $\playertwo$ maintains gives
    \[
        (\A\restriction_{\sigma'},a_0,\dots,a_{n-1}) \equiv^{\varepsilon,\Delta}_0 (\B\restriction_{\sigma'},b_0,\dots,b_{n-1}),
    \]
    and since every atomic $\sigma'$-formula is a $\Delta$-formula, this is the same as
    \[
        (\A\restriction_{\sigma'},a_0,\dots,a_{n-1}) \equiv^{\varepsilon}_0 (\B\restriction_{\sigma'},b_0,\dots,b_{n-1}),
    \]
    which is the winning condition of the game.
\end{proof}

\begin{remark}
    There is an infinite $\sigma$ and $\sigma$-structures $\A$ and $\B$ such that $\A\equiv\B$ but $\playerone\wins\EF{}1{\A,\B}$, making the restriction to finite pieces of the signature necessary. First of all, this is true classically: let $\sigma = \{P_n \mid n<\omega\}$, where each $P_n$ is a unary predicate, let $\A$ be a structure such that $P_n^\A\supseteq P_{n+1}^\A$, $|P_n^\A\setminus P_{n+1}^\A|=\aleph_0$ and $\bigcap_{n<\omega}P_n^\A$ is a singleton, and let $\B$ be the same structure as $\A$ but with the element of $\bigcap_{n<\omega}P_n^\A$ removed. Now $\playerone$ wins the classical EF game between $\A$ and $\B$ in one move by playing $a_0 \in\bigcap_{n<\omega}P_n^\A$; now $a_0\in P_n^\A$ for all $n$ but no response $b_0\in\B$ by $\playertwo$ will satisfy $b_0\in P_n^\B$ for every $n$. However, $\playertwo$ wins the game of length $n$ between the reducts of $\A$ and $\B$ to any finite $\sigma'\subseteq\sigma$ because she will only need to make sure that $y_i$ is in the same fixed finitely many relations as $x_i$. Finally, this example translates to the continuous setting by considering the structures as discrete metric structures.
\end{remark}

\section{Applications of the EF games}

Compact structures in metric model theory correspond to finite structures in classical model theory. If $\A$ is a finite classical structure and $\B$ is elementarily equivalent to $\A$, then they have the same cardinality. This translates to the metric world in the following sense: if $\varepsilon>0$ and $n$ is the least number such that $\A$ can be covered by sets of diameter $\leq\varepsilon$, then this is true also for any $\B\models\Th(\A)$ (hence making also $\B$ compact). This is easy enough to see and we formulate it as the following lemma.

\begin{lemma}\label{The theory of a compact structure enforces compactness}
    Let $\A$ be a compact structure and $\B\models\Th(\A)$. Then
    \begin{enumerate}
        \item also $\B$ is compact and
        \item for any $\varepsilon>0$, if $n<\omega$ is the least such that $\A$ can be covered by $n$ sets of diameter $<\varepsilon$, then $n$ is also the least such number for $\B$.
    \end{enumerate}
\end{lemma}
\begin{proof}
    We show that if $\A$ can be covered by $n$-many sets of diameter $<\varepsilon$, then so can $\B$. This shows both the compactness of $\B$ and that the number needed to cover $\B$ with sets of diameter $<\varepsilon$ is at most the same as for $\A$. So let $n$ be such. Then the value of the sentence
    \[
        \inf_{x_0}\dots\inf_{x_{n-1}}\sup_{y}\min_{i<n}d(y,x_i)
    \]
    in $\A$ is some number strictly smaller than $\varepsilon$, whence this number is also the value of the sentence in $\B$.
    Hence there are $b_0,\dots,b_{n-1}\in\B$ such that for all $b\in\B$, one of the numbers $d(b,b_i)$ is strictly smaller than $\varepsilon$.

    Since also $\A\models\Th(\B)$, by a symmetric argument if $\B$ can be covered by $n$-many sets of diameter $<\varepsilon$, then so can $\A$.
\end{proof}

Even more is true classically: the first-order theory of a finite structure is categorical, i.e. it has a unique model up to isomorphism. This is because a finite structure has only finitely many elements, relations and functions, so one can cook up a sentence that exactly describes each of these.\footnote{Of course, if the signature is not finite, one needs infinitely many sentences in order to describe which symbol gets which interpretation.} The same is true for the continuous first-order theory of a compact structure, which we can prove using our EF games.

This is not a new result, per se: if $\A$ and $\B$ are elementarily equivalent compact structures, then by the $\CFO$ analogue of the Keisler--Shelah theorem, there is a cardinal $\kappa$ and an ultrafilter $U$ on $\kappa$ such that $\A^\kappa/U\cong\B^\kappa/U$, but as $\A$ and $\B$ are compact, we have $\A\cong\A^\kappa/U$ and $\B\cong\B^\kappa/U$. However, we feel that the EF game is much less of a sledgehammer than the Keisler--Shelah theorem and our proof may give better insight as to why the theorem is true.

\begin{lemma}\label{We can assume that the signature is countable}\quad
    \begin{enumerate}
        \item Suppose $\A$ and $\B$ are compact $\sigma$-structures. Then there is a countable $\sigma'\subseteq\sigma$ such that $\A\cong\B$ if and only if $\A\restriction_{\sigma'}\cong\B\restriction_{\sigma'}$.
        \item Suppose $\A$ and $\B$ are $\sigma$-structures. Then there is a relational $\sigma'$ with $|\sigma'|=|\sigma|$ and $\sigma'$-structures $\A'$ and $\B'$ such that $\A\cong\B$ if and only if $\A'\cong\B'$. Furthermore, $\A'\restriction_{\emptyset} = \A\restriction_{\emptyset}$ and $\B'\restriction_{\emptyset} = \B\restriction_{\emptyset}$.
    \end{enumerate}
\end{lemma}
\begin{proof}
    \newcommand{\Fun}{\mathop{\mathrm{Fun}}}
    \newcommand{\Rel}{\mathop{\mathrm{Rel}}}
    \newcommand{\Con}{\mathop{\mathrm{Con}}}
    \begin{enumerate}
        \item For a signature $\sigma$ and structure $\A$, denote
        \begin{align*}
            \Fun(\sigma,\A) &= \{F^\A \mid \text{$F\in\sigma$ is a function symbol}\}, \\
            \Rel(\sigma,\A) &= \{P^\A \mid \text{$P\in\sigma$ is a predicate symbol}\} \text{ and} \\
            \Con(\sigma,\A) &= \{c^\A \mid \text{$c\in\sigma$ is a constant symbol} \}.
        \end{align*}
        The first two are metric spaces when equipped with the $\sup$-metric, and the third is a subspace of $\A$.
    
        First of all, suppose that $\sigma'\subseteq\sigma$ is such that the sets $\Fun(\sigma',\A)$, $\Rel(\sigma',\A)$ and $\Con(\sigma',\A)$ are dense in $\Fun(\sigma,\A)$, $\Rel(\sigma,\A)$ and $\Con(\sigma,\A)$, respectively, and the same holds for $\B$. Now $\A\restriction_{\sigma'}\cong\B\restriction_{\sigma'}$ implies $\A\cong\B$: as the interpretations of symbols of $\sigma'$ are dense in the interpretations of symbols of $\sigma$ in both structures, this means that the interpretation of every symbol of $\sigma$ is definable in both $\A\restriction_{\sigma'}$ and $\B\restriction_{\sigma'}$ and hence preserved by isomorphisms between the two.
    
        Then suppose $\A$ and $\B$ are compact. Then they are both separable, and hence so are the spaces $\Con(\sigma,\A)$ and $\Con(\sigma,\B)$. A standard Stone--Weierstrass argument can be applied to show that also the spaces $\Fun(\sigma,\A)$, $\Fun(\sigma,\B)$, $\Rel(\sigma,\A)$ and $\Rel(\sigma,\B)$ must be separable. Hence we can find a countable $\sigma'$ such that the interpretations of $\sigma'$ are dense in the set of interpretations of symbols of $\sigma$ in both structures.

        \item Let $\sigma'$ contain the constant and predicate symbols of $\sigma$ and for every $n$-ary function symbol $F\in\sigma$, we add an $n+1$-ary predicate symbol $P_F\in\sigma'$. Now we let $\A'$ and $\B'$ interpret constant and relation symbols as $\A$ and $\B$ do and let $P_F^{\A'}(a_0,\dots,a_n) = d(a_n,F^\A(a_0,\dots,a_{n-1}))$, and similarly for $\B'$. Now
        \[
            F^\A(a_0,\dots,a_{n-1}) = a_n \iff P_F^{\A'}(a_0,\dots,a_n) = 0.
        \]
        Clearly $\A'\cong\B'$ if and only if $\A\cong\B$.
    \end{enumerate}
\end{proof}

\begin{theorem}\label{The theory of a compact structure is categorical}
    Let $\A$ be a compact structure. Then $\Th(\A)$ is categorical.
\end{theorem}
\begin{proof}
    Denote by $\sigma$ the signature of $\A$, and let $\B$ be a $\sigma$-structure such that $\B\models\Th(\A)$. We show that $\B\cong\A$. By Lemma~\ref{The theory of a compact structure enforces compactness}, $\B$ is compact. By Lemma~\ref{We can assume that the signature is countable}, we may assume that $\sigma$ is countable and relational. Now, we are in a situation where both $\A$ and $\B$ are separable, so clearly it is enough to build a partial isomorphism $\A\to\B$ whose domain is dense in $\A$ and range dense in $\B$. Let $\{a_i \mid i<\omega\}$ and $\{b_i \mid i<\omega\}$ be dense in $\A$ and $\B$, respectively. By induction on $i<\omega$, we build finite partial isomorphisms $f_i,g_i\colon\A\to\B$ such that
    \begin{enumerate}
        \item for $i<j<\omega$, $f_i\subseteq g_j\subseteq f_j$,
        \item for all $i<\omega$, $a_i\in\dom(g_{i+1})$ and $b_i\in\ran(f_{i+1})$, and
        \item for $h\in\{f_i,g_i\}$, denoting $\dom(h) = \{\hat a_0,\dots,\hat a_{n-1}\}$, we have
        \[
            (\A,\hat a_0,\dots,\hat a_{n-1}) \equiv (\B, h(\hat a_0),\dots,h(\hat a_{n-1})).
        \]
        for all $k<\omega$ and finite $\sigma'\subseteq\sigma$.
    \end{enumerate}
    At the end, we let $f = \bigcup_{i<\omega}f_i = \bigcup_{i<\omega}g_i$. The third condition makes sure that this is a partial isomorphism and the second that both $\dom(f)$ and $\ran(f)$ are dense in the respective structures. First of all, we set $f_0 = g_0 = \emptyset$. Since $\A\equiv\B$ by assumption, $f_0$ and $g_0$ are as desired.
    Then suppose we have built $f_i$ and $g_i$. Let $\hat a_0,\dots,\hat a_{n-1}$ enumerate $\dom(f_i)$. Now, by the induction hypothesis and Theorem~\ref{Main theorem},
    \[
        \playertwo\wins\EF{}{k}{(\A\restriction_{\sigma'},\hat a_0,\dots,\hat a_{n-1}), (\B\restriction_{\sigma'}, f_i(\hat a_0),\dots,f_i(\hat a_{n-1}))}
    \]
    for all $k<\omega$ and finite $\sigma'$. Fix a sequence $\sigma_k$, $k<\omega$, of finite signatures such that $\sigma_k\subseteq\sigma_{k+1}$ and $\bigcup_{k<\omega}\sigma_k=\sigma$. Now for each $k$, we let $e_k\in\B$ be the move produced by the winning strategy of $\playertwo$ in
    \[
        \EF{1/(k+1)}{k}{(\A\restriction_{\sigma_k},\hat a_0,\dots,\hat a_{n-1}), (\B\restriction_{\sigma_k}, f_i(\hat a_0),\dots,f_i(\hat a_{n-1}))}
    \]
    as a response to $\playerone$ playing $a_i$ as his first move. As $\B$ is compact, $(e_k)_{k<\omega}$ has a convergent subsequence $(e_{k_i})_{i<\omega}$. We let $g_{i+1} = f_i\cup\{(a_i, e)\}$, where $e = \lim_{i\to\infty}e_{k_i}$. Left is to prove that
    \[
        (\A,\hat a_0,\dots,\hat a_{n-1},a_i) \equiv (\B, f_i(\hat a_0),\dots,f_i(\hat a_{n-1}),e).
    \]
    Let $\varphi(x_0,\dots,x_n)$ be a $\sigma$-formula, and let $\varepsilon>0$ be arbitrary. Let $\Theta$ be the modulus given by Lemma~\ref{Essential lemma} for $\varphi$.
    As $\varphi^\B$ is a uniformly continuous function, there is $\delta>0$ such that whenever $\bar b,\bar b'\in\B^{n+1}$ are such that $d(\bar b,\bar b')<\delta$, we have $\abs{\varphi^\B(\bar b) - \varphi^\B(\bar b')} \leq \varepsilon/2$. Now choose $i<\omega$ to be large enough that
    \begin{enumerate}
        \item $\varphi$ is a $\sigma_{k_i}$-formula,
        \item $k_i \geq \qr(\varphi)$,
        \item $\Theta(1/(k_i+1)) \leq \varepsilon/2$, and
        \item $d(e,e_{k_i})<\delta$.
    \end{enumerate}
    By the choice of $e_{k_i}$, we have
    \[
        \playertwo\wins\EF{1/(k_i+1)}{k_i}{(\A\restriction_{\sigma_{k_i}},\hat a_0,\dots,\hat a_{n-1},a_i), (\B\restriction_{\sigma_{k_i}}, f_i(\hat a_0),\dots,f_i(\hat a_{n-1}),e_{k_i})}.
    \]
    It follows from Lemma~\ref{Essential lemma} that
    \[
        \abs{\varphi^\A(\hat a_0,\dots,\hat a_{n-1},a_i) - \varphi^\B(f_i(\hat a_0),\dots,f_i(\hat a_{n-1}),e_{k_i})} \leq \Theta\left(\frac{1}{k_i+1}\right) \leq \frac{\varepsilon}{2},
    \]
    whence
    \begin{align*}
        \hspace{2em}&\hspace{-2em} \abs{\varphi^\A(\hat a_0,\dots,\hat a_{n-1},a_i) - \varphi^\B(f_i(\hat a_0),\dots,f_i(\hat a_{n-1}),e)} \\
        &\leq \abs{\varphi^\A(\hat a_0,\dots,\hat a_{n-1},a_i) - \varphi^\B(f_i(\hat a_0),\dots,f_i(\hat a_{n-1}),e_{k_i})} \\
        &\hspace{2.05em} + \abs{\varphi^\B(f_i(\hat a_0),\dots,f_i(\hat a_{n-1}),e_{k_i}) - \varphi^\B(f_i(\hat a_0),\dots,f_i(\hat a_{n-1}),e)} \\
        &\leq \frac{\varepsilon}{2} + \frac{\varepsilon}{2} \\
        &= \varepsilon.
    \end{align*}
    As $\varepsilon$ was arbitrary, we have
    \[
        \varphi^\A(\hat a_0,\dots,\hat a_{n-1},a_i) = \varphi^\B(f_i(\hat a_0),\dots,f_i(\hat a_{n-1}),e).
    \]
    
    The construction of $f_{i+1}$ is similar.
\end{proof}

Next we give a few inexpressibility results. Note that their proofs also demonstrate why it is essential that, in the winning condition of the game, we require that values of atomic formulae be preserved only up to $\varepsilon$-error.

\begin{corollary}
    For $\sigma=\emptyset$ and any $\delta\in(0,1]$, there is no $\sigma$-sentence $\varphi$ such that for any $\sigma$-structure $\A$,
    \[
        \A\models\varphi = 0 \iff \text{there are $a,b\in\A$ with $d(a,b)=\delta$}.
    \]
\end{corollary}
\begin{proof}
    We suppose that $\delta\leq 1/2$. The case for $\delta>1/2$ is similar.
    Let $\A$ and $\B$ be structures with domain $\omega$ such that
    \begin{enumerate}
        \item $d^\A(0,n) = \delta + \frac{1}{n+1}$ and $d^\A(n,m) = 1$ for all $0<n<m$,
        \item $d^\B(0,1) = \delta$, $d^\B(0,n) = \delta + \frac{1}{n}$ for $n>1$ and $d^\A(n,m) = d^\B(n,m) = 1$ for all $0<n<m$.
    \end{enumerate}
    These are discrete and hence complete, so they are well-defined $\sigma$-structures. Now for any $\varepsilon>0$, $\playerone$ is doomed to lose the EF game of any finite length of precision $\varepsilon$ between $\A$ and $\B$ because with precision $\varepsilon$ he has no way to tell apart two elements that are $\varepsilon$-close to each other using just the metric.
\end{proof}

If $\A$ is a finite metric structure with cardinality $n$, then if $\B$ is another structure with cardinality $\neq n$, there is a sentence $\varphi$ such that $\varphi^\A\neq\varphi^\B$. Yet, there is no sentence that would dictate that the cardinality of a structure is $n$. This result uses the finer structure of elementary equivalence.

\begin{corollary}\label{Cardinality even in the finite is not axiomatizable}
    For $\sigma=\emptyset$ and any $n<\omega$, there is no $\sigma$-sentence $\varphi$ such that for any $\sigma$-structure $\A$,
    \[
        \A\models\varphi = 0 \iff \abs{\A} \geq n.
    \]
\end{corollary}
\begin{proof}
    For simplicity, we prove this for $n = 3$. Suppose such a $\varphi$ exists. Let $\Theta$ be for $\varphi$ as in Lemma~\ref{Essential lemma}. Then let $\A$ be the discrete metric space $\{0,1\}$ with $d(0,1) = 1$. Now, as $\abs{\A}\neq 3$, we must have $\varphi^\A > 0$. Let $\varepsilon>0$ be such that $\Theta(\varepsilon) < \varphi^\A$. Then let $\B$ be the metric space $\{0,1,2\}$ with $d(0,1) = d(0,2) = 1$ and $d(1,2) = \varepsilon/2$. Now we play $\EF{\varepsilon}{\qr(\varphi)}{\A,\B}$. The stategy of $\playertwo$ is to pretend the elements $1$ and $2$ in $\B$ are the same element and play according to the ``isomorphism'' $0\mapsto 0$, $1\mapsto 1$. Clearly we have
    \[
        \abs{d^\A(0,1) - d^\B(0,1)} = \abs{d^\A(0,1) - d^\B(0,2)} = 0 < \varepsilon
    \]
    and
    \[
        \abs{d^\A(1,1) - d^\B(1,2)} = d^\B(1,2) = \frac{\varepsilon}{2} < \varepsilon.
    \]
    Hence $\playertwo$ wins. But by Lemma~\ref{Essential lemma}, this means that
    \[
        \Theta(\varepsilon) < \varphi^\A = \abs{\varphi^\A - \varphi^\B} \leq \Theta(\varepsilon),
    \]
    a contradiction.
\end{proof}

\section{Infinite and Dynamic Games}\label{Section: Infinite and Dynamic Games}

Having replicated the connection between first-order logic and EF games of finite length in the metric setting, the next natural question is what kind of infinitary logic the game of length $\omega$ would correspond to. Classically, the second player has a winning strategy in an EF game of length $\omega$ between two structures $\A$ and $\B$ if and only if $\A$ and $\B$ satisfy the same sentences of the logic $\infinitary\infty$.

There are several different definitions for an infinitary version of $\CFO$; a comparison of their expressive power is performed in~\cite{MR3729322}. As a natural counterpart of (finitary) conjunction is the continuous connective $\max$, the natural counterpart of infinitary conjunction is supremum. If one desires an infinitary extension of $\CFO$ whose formulae are continuous, then one needs to restrict the use of supremum in some way.\footnote{Whether or not one restricts suprema for continuity reasons, one of course needs to assume that the free variables of formulae of $\Phi$ are contained in a given finite set of variables, as is the case also in the classical setting.} One such way, employed by~\cite{MR2500090}, is to assume $\sup\Phi$ is a well-formed formula only when $\Phi$ is an equicontinuous set of formulae, as witnessed by some modulus $\Delta$ which is respected by each $\varphi\in\Phi$. In~\cite{MR3142393}, no such restriction is made, resulting in a discontinuous logic.

Whether the continuity restriction be placed on suprema or not, it can be immediately noted that the game defined in Definition~\ref{Game definition} does not quite do the trick for us. Consider the signature $\sigma = \{P_i \mid i<\omega\}$, where each $P_i$ is a unary predicate with $\modulus(P_i) = \frac{1}{i+1}\id$. Then denote $\varphi_i(x) = \min\{(i+1)P_i(x), 1\}$ for $i<\omega$. These are well-defined formulae of $\CFO$ and respect the modulus $\id$. Hence the formula $\varphi = \sup_{i<\omega}\varphi_i$ is a well-defined formula of $\infinitary{\omega_1}$ in the sense of both~\cite{MR2500090} and~\cite{MR3142393}. Then, let $A_i$, $i<\omega$, be infinite sets such that $A_{i+1}\subseteq A_i$ and $A_i\setminus A_{i+1}$ is infinite for all $i<\omega$, and $\bigcap_{i<\omega}A_i$ contains a single element. Let $B_i$, $i<\omega$, be a similar decreasing sequence of infinite sets but with $\bigcap_{i<\omega}B_i = \emptyset$. Now we make $A_0$ and $B_0$ discrete $\sigma$-structures $\A$ and $\B$ by setting
\[
    P_i^\A(t) = \frac{d(t,A_i)}{i+1} \quad\text{and}\quad P_i^\B(t) = \frac{d(t,B_i)}{i+1}.
\]
Now $\playertwo\wins\EF{\varepsilon}{\omega}{\A,\B}$ for any $\varepsilon>0$: if $\playerone$ plays plays the element $a$ inhabiting $\bigcap_{i<\omega}A_i$, then $\playertwo$ responds with any $b\in B_n$ such that $2/(n+1) < \varepsilon$; then for any $i\leq n$, $a\in A_i$ and $b\in B_i$, so
\[
    \abs{P_i^\A(a) - P_i^\B(b)} = 0 < \varepsilon,
\]
and for $i>n$, we have
\[
    \abs{P_i^\A(a) - P_i^\B(b)} \leq P_i^\A(a) + P_i^\B(b) = \frac{d(a,A_i)}{i+1} + \frac{d(b,B_i)}{i+1} \leq \frac{2}{i+1} < \frac{2}{n+1} < \varepsilon.
\]
On the other hand, $(\inf_x\sup_i\varphi_i)^\A = 0$ but $(\inf_x\sup_i\varphi_i)^\B = 1$, so $\A$ and $\B$ do not even satisfy the same sentences of $\infinitary{\omega_1}$ of quantifier rank $1$.

The essence of the problem seems to be the following: preserving atomic formulae arbitrarily well is not enough to preserve the supremum arbitrarily well unless one can exactly preserve the values of atomic formulae. This is due to connectives. It would seem we need a way to restrict more than just the precision at a given play. One way to modify the game would be to put in a modulus $\Delta$ as a parameter: $\playertwo$ wins the game $\EF{\varepsilon,\Delta}{\omega}{\A,\B}$ if the function resulting from a play is a $(\varepsilon,\Delta)$-isomorphism instead of just an $\varepsilon$-isomorphism. For finite signatures, this restriction adds nothing, as is seen in the proof of Theorem~\ref{Main theorem}, but for infinite signatures it may help us find winning strategies. However, the problem posed by the above example will not go away with this addition: in the example, $\playertwo$ will have a winning strategy in any $\EF{\varepsilon,\Delta}{\omega}{\A,\B}$ for $\Delta\geq\id$.

A similar problem is encountered in~\cite{MR3689736}, where the logic of~\cite{MR2500090} is further studied and a version of Scott's isomorphism theorem is proved for separable metric structures in a countable signature. Their solution is to define a fragment of $\infinitary{\omega_1}$ such that every \emph{sentence} is equivalent to a sentence of this fragment, with possibly higher quantifier rank. We can adapt this approach to our game and get a connection between $\infinitary\infty$-equivalence and infinite games. However, it is not clear whether every $\infinitary\infty$-sentence has a normal form in this fragment, as the result of~\cite{MR3689736} employs the machinery of descriptive set theory to show this for $\infinitary{\omega_1}$, and going above $\omega_1$ would seem to require results of generalized descriptive set theory. We leave the question about the normal form open for future research.

\begin{definition}[\cite{MR3689736}]\quad
    \begin{enumerate}
        \item A function $\Delta\colon[0,\infty)^n\to[0,\infty)$ is an \emph{$n$-ary modulus} if it is non-decreasing, subadditive, continuous and vanishes at zero.
        \item An $n$-ary function $f\colon M^n\to M'$ between metric spaces \emph{obeys} or \emph{respects} an $n$-ary modulus $\Delta$ if for all $x_0,y_0,\dots,x_{n-1},y_{n-1}\in M$ we have
        \[
            d(f(x_0,\dots,x_{n-1}),f(y_0,\dots,y_{n-1})) \leq \Delta(d(x_0,y_0),\dots,d(x_{n-1},y_{n-1})).
        \]
        \item A function $\Omega\colon[0,\infty)^\omega\to[0,\infty]$ is a \emph{weak modulus} if it is non-decreasing, subadditive, lower semicontinuous (in the product topology of $[0,\infty)^\omega$), separately continuous in each argument and vanishes at zero.
        \item For $k\in\N$, we denote by $\Omega|_k$ the $k$-truncation of $\Omega$, which is a function $[0,\infty)^k\to[0,\infty]$ defined by
        \[
            \Omega|_k(x_0,\dots,x_{n-1}) = \Omega(x_0,\dots,x_{n-1}, 0, 0, \dots).
        \]
        \item An $n$-ary function respects a weak modulus $\Omega$ if it respects $\Omega|_n$.
    \end{enumerate}
\end{definition}

Note that a modulus as per our prior definition is a unary modulus in this new definition. Given a unary modulus $\Delta$ and a natural number $n$, one can always define an $n$-ary modulus $\Delta_n$ by setting
\[
    \Delta_n(t_0,\dots,t_{n-1}) = \Delta\left(\max_{i<n}t_i\right)
\]
for all $t_0,\dots,t_{n-1}\in[0,\infty)$.

We call $\varphi$ a \emph{basic formula} if it is of the form $u(\psi_0,\dots,\psi_{n-1})$ for atomic $\psi_i$ and some (finitary) connective $u$.

\begin{definition}[\cite{MR3689736}]
    Let $\Omega$ be a weak modulus and $\kappa$ a regular cardinal. The set of $\Omega$-formulae of $\infinitary\kappa$ is defined as follows.
    \begin{enumerate}
        \item Every basic formula $\varphi(v_0,\dots,v_{n-1})$ that respects $\Omega|_n$ is an $n$-ary $\Omega$-formula.
        \item If $\varphi(v_0,\dots,v_n)$ is an $n+1$-ary $\Omega$-formula, then $\sup_{v_n}\varphi$ and $\inf_{v_n}\varphi$ are $n$-ary $\Omega$-formulae.
        \item If $\{\varphi_i \mid i<\alpha\}$, $\alpha<\kappa$, is a set of $n$-ary $\Omega$-formulae, then $\sup_{i<\alpha}\varphi_i$ and $\inf_{i<\alpha}\varphi_i$ are $n$-ary $\Omega$-formulae.
        \item If $\varphi_0,\dots,\varphi_{k-1}$ are $n$-ary $\Omega$-formulae and $u\colon[0,1]^k\to[0,1]$ is a $1$-Lipschitz connective, then $u(\varphi_0,\dots,\varphi_{k-1})$ is an $n$-ary $\Omega$-formula.
    \end{enumerate}
    An $\Omega$-formula is an $n$-ary $\Omega$-formula for some $n<\omega$.
\end{definition}

Notice that there are no cumulative restrictions on subformulae: if $u(\varphi)$ is a basic $\Omega$-formula for atomic $\varphi$, it is not required that $\varphi$ respects $\Omega$. Furthermore, no connective added after the initial step of the definition will ``steepen'' the formula. Thus all the steepness acquired from connectives comes from this initial connective. Notice also that, in an $n$-ary $\Omega$-formula, the free variables are among the first $n$ variable symbols, and quantification is only allowed over the free variable of greatest index.

Now we define a version of the EF game parametrized on a weak modulus $\Omega$. We also define dynamic versions that have an ordinal clock.

\begin{definition}
    Let $\A$ and $\B$ be disjoint structures of the same signature, and let $\bar a\in\A^n$ and $\bar b\in\B^n$.
    \begin{enumerate}
        \item Given $\varepsilon>0$ and a weak modulus $\Omega$, the $(\varepsilon,\Omega)$-Ehrenfeucht--Fraïssé game of length $\omega$ between $(\A,a_0,\dots,a_{n-1})$ and $(\B,b_0,\dots,b_{n-1})$, denoted by
        \[
            \EF{\varepsilon,\Omega}{\omega}{(\A,a_0,\dots,a_{n-1}),(\B,b_0,\dots,b_{n-1})},
        \]
        is played exactly as follows. On each round $i<\omega$, $\playerone$ chooses an element $x_{n+i}\in\A\cup\B$ and $\playertwo$ responds by choosing an element $y_{n+i}\in\A\cup\B$ such that $x_{n+i}\in\A$ if and only if $y_{n+i}\in\B$. A play ends only after $\omega$-many rounds. Then $\playertwo$ wins a play if for all $k<\omega$ and all \emph{basic} $k$-ary $\Omega$-formulae $\varphi(v_0,\dots,v_{k-1})$, we have
        \[
            \abs{\varphi^\A(a_0,\dots,a_{k-1}) - \varphi^\B(b_0,\dots,b_{k-1})} \leq \varepsilon.
        \]
    
        \item Given $\varepsilon>0$ and a weak modulus $\Omega$, the dynamic $(\varepsilon,\Omega)$-Ehrenfeucht--Fraïssé game with clock $\alpha$ between $(\A,\bar a)$ and $(\B,\bar b)$, denoted by
        \[
            \EFD{\alpha}{\varepsilon,\Omega}{(\A,\bar a),(\B,\bar b)},
        \]
        is a version of $\EF{\varepsilon,\Omega}{\omega}{(\A,\bar a),(\B,\bar b)}$ where the length of a play is not infinite but, instead, dynamic: on each round $i$, in addition to the element $x_{n+i}$, $\playerone$ also plays an ordinal $\alpha_i<\alpha$ such that for all $i<j<\omega$, $\alpha_j < \alpha_i$, and a play ends after the round $i$ when $\playerone$ chooses $\alpha_i = 0$.

        \item $\EF{\Omega}{\omega}{(\A,\bar a),(\B,\bar b)}$ is the game where $\playerone$ plays some $\varepsilon>0$ on the first round, in addition to an element of $\A\cup\B$, and then players proceed as if they were playing $\EF{\varepsilon,\Omega}{\omega}{(\A,\bar a),(\B,\bar b)}$. $\EFD{\alpha}{\Omega}{(\A,\bar a),(\B,\bar b)}$ is defined similarly.
    \end{enumerate}
\end{definition}

An important detail about the winning condition is that given a basic $n$-ary $\Omega$-formula $\varphi(v_0,\dots,v_{n-1})$, the variable $v_i$ is always interpreted as $a_i$ in $\A$ and $b_i$ in $\B$. From this it follows that, for instance, one can compare $a_3$ and $b_3$ only using formulae that obey $\Omega|_3$ but $a_{42}$ and $b_{42}$ can be compared using formulae that obey $\Omega|_{42}$. Often $\Omega|_n$ is less restricting than $\Omega|_m$ for $m\leq n$. Another thing to note is that $\EF{\varepsilon,\Omega}{\omega}{(\A,\bar a),(\B,\bar b)}$ (and $\EFD{\alpha}{\varepsilon,\Omega}{(\A,\bar a),(\B,\bar b)}$, respectively) is \emph{not} a game between expansions $\hat\A$ and $\hat\B$ of $\A$ and $\B$ by interpretations $a_i$ and $b_i$ of new constant symbols $c_i$. Rather, the starting position of a play of $\EF{\varepsilon,\Omega}{\omega}{(\A,\bar a),(\B,\bar b)}$ corresponds to a play of $\EF{\varepsilon,\Omega}{\omega}{\A,\B}$ where $n$ rounds have already been played and the current position is $((a_0,b_0),\dots,(a_{n-1},b_{n-1}))$.

We now connect the dynamic games to pseudometrics defined in~\cite{MR3689736} to measure how ``far'' two structures are from satisfying the same $\Omega$-conditions of $\infinitary\infty$ of a given quantifier rank.

\begin{definition}[\cite{MR3689736}]
    For a weak modulus $\Omega$ and two structures $\A$ and $\B$, we define functions $r_{\alpha,n}^{\A,\B,\Omega}\colon\A^n\times\B^n\to[0,1]$ for each ordinal $\alpha$ as follows.
    \begin{enumerate}
        \item If $\alpha=0$, then
        \[
            r_{\alpha,n}^{\A,\B,\Omega}(\bar{a},\bar{b}) = \sup\{|\varphi^\A(\bar{a}) - \varphi^\B(\bar{b})| : \text{$\varphi$ is a basic $n$-ary $\Omega$-formula}\}.
        \]
        \item If $\alpha = \beta+1$, then
        \[
            r_{\alpha,n}^{\A,\B,\Omega}(\bar{a},\bar{b}) = \max\{\sup_{a\in\A}\inf_{b\in\B}r_{\beta,n+1}^{\A,\B,\Omega}(\bar{a}a,\bar{b}b),\ \sup_{b\in\B}\inf_{a\in\A}r_{\beta,n+1}^{\A,\B,\Omega}(\bar{a}a,\bar{b}b)\}.
        \]
        \item If $\alpha$ is a limit, then
        \[
            r_{\alpha,n}^{\A,\B,\Omega}(\bar{a},\bar{b}) = \sup_{\beta<\alpha}r_{\beta,n}^{\A,\B,\Omega}(\bar{a},\bar{b}).
        \]
    \end{enumerate}
    Then we let $r_{\infty,n}^{\A,\B,\Omega}(\bar{a},\bar{b})$ denote the supremum of $r_{\alpha,n}^{\A,\B,\Omega}(\bar{a},\bar{b})$ over all ordinals $\alpha$. We may omit $n$ and $\Omega$ from the notation when they are clear from the context. We may also write $r_\alpha(\A\bar{a},\B\bar{b})$ for $r_{\alpha}^{\A,\B}(\bar{a},\bar{b})$.
\end{definition}

\begin{fact}[\cite{MR3689736}]\quad
    \begin{enumerate}
        \item $r_{\alpha,n}$ is a pseudometric on all pairs $(\A,\bar a)$ for $\A$ a structure and $\bar a\in\A^n$.
        \item If $\beta\leq\alpha$, then $r_\beta\leq r_\alpha$.
        \item For all $\A$ and $\B$, there is $\alpha<(\density(\A)+\density(\B))^+$ such that for all $\beta>\alpha$, $r_\alpha^{\A,\B} = r_\beta^{\A,\B} = r_\infty^{\A,\B}$.
        \item For all $\alpha$, $\A$, $\B$, $n$, $\bar a\in\A^n$ and $\bar b\in\B^n$,
        \[
            \hspace{2em} r_{\alpha,n}^{\A,\B,\Omega}(\bar{a},\bar{b}) = \sup\{|\varphi^\A(\bar{a}) - \varphi^\B(\bar{b})| : \text{$\varphi$ an $n$-ary $\Omega$-formula with $\qr(\varphi)\leq\alpha$}\}.
        \]
    \end{enumerate}
\end{fact}

\begin{theorem}\label{dynamic games vs r-distances}
    Let $\Omega$ be a weak modulus. Then for all ordinals $\alpha$, structures $\A$ and $\B$ and tuples $\bar a\in\A^n$ and $\bar b\in\B^n$, we have
    \[
        r_{\alpha,n}^{\A,\B,\Omega}(\bar{a},\bar{b}) = \inf\{\varepsilon\in(0,1] \mid \playertwo\wins\EFD{\alpha}{\varepsilon,\Omega}{(\A,\bar{a}),(\B,\bar{b})} \}.
    \]
\end{theorem}
\begin{proof}
    We begin by proving that $r_{\alpha,n}^{\A,\B}(\bar{a},\bar{b})$ is a lower bound of the set
    \[
        \{\varepsilon\in(0,1] \mid \playertwo\wins\EFD{\alpha}{\varepsilon,\Omega}{(\A,\bar{a}),(\B,\bar{b})} \}.
    \]
    We show by induction on $\alpha$ that if $\playertwo\wins\EFD{\alpha}{\varepsilon,\Omega}{(\A,\bar{a}),(\B,\bar{b})}$, then $r_{\alpha,n}^{\A,\B}(\bar{a},\bar{b})\leq\varepsilon$. The cases $\alpha = 0$ and $\alpha$ limit are clear, so we consider the case $\alpha = \beta+1$. Suppose that $\playertwo\wins\EFD{\alpha}{\varepsilon,\Omega}{(\A,\bar{a}),(\B,\bar{b})}$. By definition, 
    \[
        r_{\alpha,n}^{\A,\B}(\bar{a},\bar{b}) = \max\{\sup_{a\in\A}\inf_{b\in\B}r_{\beta,n+1}^{\A,\B}(\bar{a}a,\bar{b}b),\ \sup_{b\in\B}\inf_{a\in\A}r_{\beta,n+1}^{\A,\B}(\bar{a}a,\bar{b}b)\}.
    \]
    By symmetry, it is enough to show that $\sup_{a\in\A}\inf_{b\in\B}r_{\beta,n+1}^{\A,\B}(\bar{a}a,\bar{b}b)\leq\varepsilon$. Let $a\in\A$ be arbitrary. We now play $\EFD{\alpha}{\varepsilon,\Omega}{(\A,\bar{a}),(\B,\bar{b})}$. Let $(a,\beta)$ be the first move by $\playerone$. Then the winning strategy of $\playertwo$ produces some $b\in\B$ such that $\playertwo\wins\EFD{\beta}{\varepsilon,\Omega}{(\A,\bar{a}a),(\B,\bar{b}b)}$. By the induction hypothesis, this means that $r^{\A,\B}_{\beta,n+1}(\bar aa,\bar bb) \leq \varepsilon$. As for any $a\in\A$ there is $b\in\B$ such that $r^{\A,\B}_{\beta,n+1}(\bar aa,\bar bb) \leq \varepsilon$, we have $\sup_{a\in\A}\inf_{b\in\B}r_{\beta,n+1}^{\A,\B}(\bar{a}a,\bar{b}b)\leq\varepsilon$.

    To prove that $r_{\alpha,n}^{\A,\B}(\bar{a},\bar{b})$ is the greatest lower bound of the particular set, suppose $\varepsilon>r_{\alpha,n}^{\A,\B}(\bar{a},\bar{b})$. We show that $\playertwo\wins\EFD{\alpha}{\varepsilon,\Omega}{(\A,\bar{a}),(\B,\bar{b})}$. We again proceed by induction on $\alpha$, and the cases $\alpha = 0$ and $\alpha$ limit are clear. Suppose $\alpha = \beta+1$, and let us play $\EFD{\alpha}{\varepsilon,\Omega}{(\A,\bar{a}),(\B,\bar{b})}$. Note that since $r_\beta \leq r_\alpha$, by the induction hypothesis $\playertwo\wins\EFD{\beta}{\varepsilon,\Omega}{(\A,\bar{a}),(\B,\bar{b})}$. So if $\playerone$ plays $\alpha_0 < \beta$ in his first move, then $\playertwo$ can use her winning strategy in the $\beta$-game to win. We may thus assume that $\playerone$ plays $\alpha_0 = \beta$. Let $a\in\A$ be the element $\playerone$ plays on the first round (the case where $\playerone$ plays in $\B$ is symmetric). Now, as $r_{\alpha,n}^{\A,\B}(\bar a,\bar b) < \varepsilon$, we have $\sup_{a'\in\A}\inf_{b\in\B}r_{\beta,n+1}^{\A,\B}(\bar{a}a',\bar{b}b)<\varepsilon$. Hence, in particular, $\inf_{b\in\B}r_{\beta,n+1}^{\A,\B}(\bar{a}a,\bar{b}b)<\varepsilon$, so there is $b\in\B$ with $r_{\beta,n+1}^{\A,\B}(\bar{a}a,\bar{b}b)<\varepsilon$. We let $\playertwo$ play $b$ as a response on the first round. By the induction hypothesis, $\playertwo\wins\EFD{\beta}{\varepsilon,\Omega}{(\A,\bar{a}a),(\B,\bar{b}b)}$, so following her strategy in this $\beta$-game, $\playertwo$ survives the rest of the play in the $\alpha$-game.
\end{proof}

We write $\A\equiv_\alpha^\Omega\B$ if $\varphi^\A = \varphi^\B$ for all $\Omega$-sentences of quantifier rank $\leq\alpha$. If $\A\equiv_\alpha^\Omega\B$ for all $\alpha$, we write $\A\equiv_\infty^\Omega\B$.

\begin{corollary}
    The following are equivalent for all structures $\A$ and $\B$, ordinals $\alpha$ and non-negative reals $\varepsilon$.
    \begin{enumerate}
        \item $\playertwo\wins\EFD{\alpha}{\delta,\Omega}{(\A,\bar a),(\B,\bar b)}$ for all $\delta>\varepsilon$.
        \item $\abs{\varphi^\A(\bar a) - \varphi^\B(\bar b)}\leq\varepsilon$ for all $n$, $\bar a\in\A^n$ and $\bar b\in\B^n$.
    \end{enumerate}
    In particular, $\playertwo\wins\EFD{\alpha}{\Omega}{\A,\B}$ if and only if $\A\equiv_\alpha^\Omega\B$.
\end{corollary}

\begin{lemma}
    $\playertwo\wins\EF{\varepsilon,\Omega}{\omega}{(\A,\bar a),(\B,\bar b)}$ if and only if $\playertwo\wins\EFD{\alpha}{\varepsilon,\Omega}{(\A,\bar a),(\B,\bar b)}$ for all ordinals $\alpha$.
\end{lemma}
\begin{proof}
    Similar to the classical proof, see e.g.~\cite{Vmodels}.
\end{proof}

\begin{corollary}
    $\playertwo\wins\EF{\Omega}{\omega}{\A,\B}$ if and only if $\A\equiv_\infty^\Omega\B$.
\end{corollary}

It is shown in~\cite{MR3689736} that for any countable signature $\sigma$, there is a modulus $\Omega_\sigma$ such that every $\sigma$-sentence of $\infinitary{\omega_1}$ is equivalent to an $\Omega_\sigma$-sentence. From this we obtain the following.

\begin{corollary}
    $\playertwo\wins\EFD{\alpha}{\Omega_\sigma}{\A,\B}$ for all $\alpha<\omega_1$ if and only if $\A$ and $\B$ satisfy the same sentences of $\infinitary{\omega_1}$.
\end{corollary}

It remains as an open question whether $\playertwo\wins\EF{\Omega_\sigma}{\omega}{\A,\B}$ if and only if $\A$ and $\B$ satisfy the same sentences of $\infinitary\infty$. This is tied to the open question about whether every formula of $\infinitary\infty$ is equivalent to an $\Omega_\sigma$-formula. For uncountable signatures $\sigma$, the existence of a universal weak modulus $\Omega_\sigma$ seems unlikely.

\bibliographystyle{alpha}
\bibliography{arxiv_version2}

\end{document}